\documentclass[a4paper,11pt]{article}
\usepackage[utf8]{inputenc}
\usepackage{amsmath,amsthm,mathrsfs}
\usepackage{latexsym}
\usepackage{amssymb}
\usepackage[all]{xy}
\usepackage{enumerate}
\usepackage{pst-all}
\usepackage{pstricks}
\usepackage{enumerate}
\usepackage{lipsum}
\usepackage{float}
\usepackage{multicol}
\input{xy}
\theoremstyle{definition}
 \newtheorem{theo}{Theorem}[section]
 \newtheorem{propo}[theo]{Proposition}
 \newtheorem{lem}[theo]{Lemma}
 
 \newtheorem{obs}[theo]{Observation}
 \newtheorem{ejem}[theo]{Example}
 \newtheorem{defin}[theo]{Definition}
\def\L{\Lambda}
\def\G{\Gamma}
\def\oo{\mathfrak{o}}

\def\C{\mathscr{C}}

\def\T{\mathscr{T}}
\def\V{\mathscr{V}}
\def\Q{\mathcal{Q}}
\def\CM{\mathcal{C}}

\def\occ{\textrm{occ}}
\def\val{\textrm{val}}
\def\CC{\mathscr{CC}}
\def\hom{\textrm{Hom}}
\def\enD{\textrm{End}}

\def\rad{\textrm{rad}}
\def\top{\textrm{top}}
\title{The Cartan Matrix of a Brauer Configuration Algebra}
\author{Alex Sierra C.}
\date{}
\begin{document}
 
 \maketitle
 
 \begin{abstract}
  Using the combinatorial information of a Brauer configuration is possible to compute each entry of the Cartan matrix of the algebra associated to the Brauer configuration.
 \end{abstract}
\section{Introduction}\label{inttro01}
%Brauer configuration algebras were recently added to the mathematical literature by E. Green and S. Schroll in \cite{brau} as a generalization of Brauer graph algebras. As the authors mention in \cite[Introduction]{brau} this class of algebras is a new class of mostly wild algebras whose additional structure arises from a combinatorial data, called Brauer configuration. Brauer configuration algebras are a generalization of Brauer graph algebras, in the sense that every Brauer graph is a Brauer configuration and every Brauer graph algebra is a Brauer configuration algebra.\\

Recently was calculated the dimension of the center of a Brauer configuration algebra using the combinatorial information of the configuration only \cite[Theorem 4.9]{mya}. In the process of this calculation is strongly used the following result.

\begin{propo}\cite[Proposition 3.3]{mya}\label{001}
 Let $\L=K\Q/I$ be the Brauer configuration algebra induced by $\G=(\G_0,\G_1,\mu,\oo)$ and let $V\in\G_1$. If $v$ is the vertex in $\Q$ associated to the polygon $V$ then \[\textrm{dim}_Kv\L v=2+\sum_{\alpha\in V\cap\G_0}\occ(\alpha,V)(\occ(\alpha,V)\mu(\alpha)-1).\]
\end{propo}

In the proof of Proposition \ref{001} is implemented a combinatorial technique based on the fact that every basis element of the induced Brauer configuration algebra is the class of a prefix of an element of the form $C^{\mu(\alpha)}$, where $\alpha$ is a nontruncated vertex and $C$ is a special $\alpha$-cycle \cite[Proposition 3.3]{brau}. The combinatorial technique is fully explained in \cite[Section 3]{mya}, and its main tool is what we call here a special and non-special diagram associated to a nontruncated vertex. We may say that this diagram represents the graphical information of the orientation $\oo$ for $\G$ when is looked inside the induced quiver of $\G$. It is the flow course of the orientation throughout the arrows.

In Section \ref{sec3} we introduce the $(\alpha,v)$-intervals of a special and non-special diagram. These intervals are formed by dividing the diagram associated to $\alpha$ in intervals according to the occurrences of the vertex $v$ on the diagram. The use of these intervals is crucial in the proof of Proposition \ref{018}, which gives us the value of all the entries that are out of the main diagonal of the Cartan matrix of any Brauer configuration algebra. We will see the power and the information that can be obtained by making a very simple graphical interpretation of the orientation of a Brauer configuration.

\section{Some preliminaries about Brauer configuration algebras}\label{sec02}

%In this section we start introducing the Brauer configurations to define later a Brauer configuration algebra. Then we recall the construction and definitions of \cite[Section 2]{mya} to prove the principal result. For a more explicit a complete presentation of the Brauer configuration algebras, and lots of examples, we refer the reader to \cite{brau}. To make easier the reading we will use the same notation implemented in \cite{mya}.
%\\

A \textit{Brauer configuration} is a tuple $\G=(\G_0,\G_1,\mu,\oo)$ such that
\begin{enumerate}[(1)]
\item $\G_0$ is a finite set of elements that we call \textit{vertices};
\item $\G_1$ is a finite collection of finite labeled multisets\footnote{A multiset is a set where repetitions of elements are allowed.} each of them formed by vertices. We call each element of $\G_1$ a \textit{polygon};
\item $\mu:\G_0\to\mathbb{Z}_{>0}$ is a set function that we call the \textit{multiplicity function};
\item for $\alpha\in\G_0$ and $V\in\G_1$ we denote by $\occ(\alpha,V)$ the number of times that the vertex $\alpha$ appears in $V$. By $\val(\alpha)$ we denote the  integer value \[\val(\alpha):=\sum_{V\in\G_1}\occ(\alpha,V).\] A vertex $\alpha\in\G_0$ is called \textit{truncated} if $\val(\alpha)=\mu(\alpha)=1$. The \textit{orientation} $\oo$ means that for each nontruncated vertex $\alpha$ there is a chosen cyclic ordering of the polygons that contain $\alpha$, including repetitions. See Observation \ref{003}.
\end{enumerate}

Additionally to this we require that $\G$ satisfies
 \begin{enumerate}
  \item[C1.] Every vertex in $\G_0$ is a vertex in at least one polygon in $\G_1$.
  \item[C2.] Every polygon in $\G_1$ has at least two vertices (which can be the same).
  \item[C3.] Every polygon in $\G_1$ has at least one vertex $\alpha$ such that $\val(\alpha)\mu(\alpha)>1$.
 \end{enumerate}

\begin{obs}\label{003}
For each $\alpha\in\G_0$ such that $\val(\alpha)=t>1$ or $\mu(\alpha)>1$, let $V_1,\ldots, V_t$ be the list of polygons in which $\alpha$ occurs as a vertex, with a polygon $V$ occurring $\occ(\alpha,V)$ times in the list, that is $V$ occurs the number of times $\alpha$ occurs as a vertex in $V$. The cyclic order at vertex $\alpha$ is obtained by linearly ordering the list, say $V_{i_1}<\cdots<V_{i_t}$ and by adding $V_{i_t}<V_{i_1}$. We observe that any cyclic permutation of a chosen cyclic ordering at vertex $\alpha$ can represent the same ordering. That is, if $V_1<\cdots<V_t$ is the chosen cyclic ordering at vertex $\alpha$, so is a cyclic permutation such as $V_2<V_3<\cdots<V_t<V_1$ or $V_3<V_4<\cdots<V_t<V_1<V_2$. If $V_{i_1}<\cdots<V_{i_t}$ is a cyclic order at the vertex $\alpha$ we denote this by $\alpha:V_{i_1}<\cdots<V_{i_t}$ and we call it a \textit{successor sequence at} $\alpha$. Now, if we have that $\mu(\alpha)>1$ but $\val(\alpha)=1$ and $V$ is the only polygon where $\alpha$ belongs, we denote the successor sequence at $\alpha$ simply as $\alpha:V$.
\end{obs}
\begin{ejem}\label{002}
Let $\G=(\G_0,\G_1,\mu,\oo)$ be the Brauer configuration given by the following data: $\G_0=\{1,2,3,4\},\G_1=\{V_1,V_2,V_3,V_4\}$ where $V_1=\{1,2\},V_2=\{1,2\},V_3=\{1,1,3,3\}$ and $V_4=\{3,4\}$. Observe that $V_1$ and $V_2$ have both the same set of vertices, however they are considered as different polygons in the configuration. If we define the multiplicity function as $\mu(3)=\mu(4)=1$ and $\mu(1)=\mu(2)=2$, we see that the vertex 4 is truncated. For the orientation $\oo$ we chose the following successor sequences.\[\begin{array}{rcl}1 & : & V_1<V_2<V_3<V_3;\\2 & : & V_1<V_2;\\3 & : & V_3<V_4<V_3.\end{array}\] Also observe that in this orientation, and any other chosen orientation, the number of times that appears a polygon in a successor sequence must coincide with the number of times that the associated vertex appears in the polygon.
\end{ejem}
For $\alpha\in\G_0$ a nontruncated vertex wiht $\val(\alpha)>1$, let $\alpha:V_{i_1}<\cdots<V_{i_t}$ be a successor sequence of $\alpha$, where $t=\val(\alpha)$. We say that $V_{i_{j+1}}$ is the \textit{successor} of $V_{i_j}$, for all $1\le j\le t$, and $V_{j_{t+1}}=V_{i_1}$. If $\val(\alpha)=1$ but $\mu(\alpha)>1$, and $V$ is the only polygon where $\alpha$ belongs, then we say that $V$ is its own successor at $\alpha$. We can also have that a polygon is its own successor at a vertex $\alpha$ where $\val(\alpha)>1$. This is the case for the polygon $V_3$ in the Example \ref{002} above. It is its own successor at the vertex 1.

Now, we define the induced quiver by a Brauer configuration. For $\G$ a Brauer configuration let $\Q$ the quiver induced by
\begin{itemize}
\item The set of vertices of $\Q$ is in one-to-one correspondence whith the set of polygons of $\G$. If $V$ is a polygon in $\G_1$, we will denote the associated vertex in $\Q$ by $v$, and we say that $v$ is the vertex in $\Q$ associated to $V$.
\item If the polygon $V'$ is a successor to the polygon $V$ at $\alpha$, there is an arrow from $v$ to $v'$, where $v$ is the vertex in $\Q$ associated to $V$, and $v'$ is the vertex in $\Q$ associated to $V'$.
 \end{itemize}
If we denote by $\T_{\G}$ the set of all truncated vertices of $\G$ and $\Q_1$ the collection of all the arrows in the induced quiver $Q$, then it is not difficult to prove that \[|\Q_1|=\sum_{\G_0\setminus\T_{\G}}\val(\alpha).\] For each nontruncated vertex $\alpha$ in $\G_0$ with $\val(\alpha)=t>1$ and successor sequence $\alpha:V_{i_1}<\cdots<V_{i_t}$, we have a corresponding sequence of arrows in the induced quiver $\Q$

\begin{equation}\label{004}
v_{i_1}\xrightarrow{a^{(\alpha)}_{j_1}} v_{i_2} \xrightarrow{a^{(\alpha)}_{j_2}} \cdots \xrightarrow{a^{(\alpha)}_{j_{t-1}}}v_{i_{t}}\xrightarrow{a^{(\alpha)}_{j_{t}}} v_{i_1}.
\end{equation}
Let $C_l=a^{(\alpha)}_{j_{l}}a^{(\alpha)}_{j_{l+1}}\cdots a^{(\alpha)}_{j_{t}}a^{(\alpha)}_{j_{1}}\cdots a^{(\alpha)}_{j_{l-1}}$ be the oriented cycle in $\Q$, for $1\le l\le t$. We call any of these cycles a \textit{special} $\alpha$-\textit{cycle}. We observe that when $\alpha$ is a nontruncated vertex such that $\val(\alpha)=1$, we have only one special $\alpha$-cycle, which is a loop at the vertex in $\Q$ associated to the unique polygon containing $\alpha$. Now, let $V$ be a fixed polygon in $\G_1$ such that $\alpha$ is a vertex in $V$ and $\occ(\alpha,V)=s\ge1$. Then there are $s$ indices $l_1,\ldots,l_s$ such that $V=V_{i_{l_r}}$, for every $1\le r\le s$. We call any of the cycles $C_{l_1},\ldots,C_{l_s}$ a \textit{special} $\alpha$-\textit{cycle at} $v$, and we denote the collection of these cycles in $\Q$ by $\C_{(\alpha)}^{\,v}$. If we denote by $\CC_{(\alpha)}$ the collection of all the special $\alpha$-cycles in $\Q$ and define $\V_{(\alpha)}=\{\,V\in\G_1\,|\,\alpha\in V\,\}$, the set of polygons containing $\alpha$, is easy to prove that
\begin{equation}\label{005}\CC_{(\alpha)}=\bigcup_{V\in\V_{(\alpha)}}\C_{(\alpha)}^{\,v}.\end{equation} Once again, for the particular case of a nontruncated vertex $\alpha$ such that $\val(\alpha)=1$, the collection $\CC_{(\alpha)}$ is just the set consisting of the unique loop at the vertex in $\Q$ associated to the unique polygon containing $\alpha$.

\begin{ejem}\label{007}
The quiver $\Q$ induced by the configuration $\G$ of Example \ref{002} is
\begin{equation}\label{006}
\begin{split}
\xymatrix{ & & v_4\ar@/_1.5pc/[dd]_{a^{(3)}_2} & & \\  & & & & \\  & &v_3\ar@/_1pc/[lldd]_{a^{(1)}_4}\ar@(ul,ur)^{a^{(3)}_3}\ar@(dr,dl)^{a^{(1)}_3}\ar@/_1.5pc/[uu]_{a^{(3)}_1} & &\\ & & & \\v_1\ar@/_2.4pc/[rrrr]_{a^{(1)}_1}\ar@/_1pc/[rrrr]^{a^{(2)}_1} & & & & v_2\ar@/_1pc/[uull]_{a^{(1)}_2}\ar@/_1pc/[llll]_{a^{(2)}_2}}
\end{split}
\end{equation}
As we can see every arrow is induced by a successor sequence. For example, the successor sequence at vertex 1 induces in $\Q$ the sequence of arrows \[1:v_1\xrightarrow{a^{(1)}_1} v_2\xrightarrow{a^{(1)}_2}v_3\xrightarrow{a^{(1)}_3}v_3\xrightarrow{a^{(1)}_4}v_1;\] and the successor sequence at 3 induces the sequence of arrows \[3:v_3\xrightarrow{a^{(3)}_1}v_4\xrightarrow{a^{(3)}_2}v_3\xrightarrow{a^{(3)}_3}v_3.\]
\end{ejem}
For the induced quiver $\Q$ of the Brauer configuration $\G=(\G_0,\G_1,\mu,\oo)$ define the set
\begin{equation}\label{008}
\CC:=\bigcup_{\alpha\in\G_0\setminus\T_{\G}}\CC_{(\alpha)},
\end{equation}
 and let $f:\CC\to\Q_1$ be the map which sends a special cycle to its first arrow. Now, for the field $K$ let's consider in the path algebra $K\Q$ the following type of relations.\\
 
\noindent\textit{Relations of type one.} It is the subset of $K\Q_{\G}$ \[\bigcup_{V\in\G_1}\left(\bigcup_{\alpha,\beta\in V\setminus\T_{\G}}\left\{\,C^{\mu(\alpha)}-D^{\mu(\beta)}\,|\,C\in\C_{(\alpha)}^{\,v},D\in\C_{(\beta)}^{\,v}\,\right\}\right).\]
\\
\textit{Relations of type two.} It is the subset of $K\Q_{\G}$ \[\bigcup_{\alpha\in\G_0\setminus\T_{\G}}\left\{\,C^{\mu(\alpha)}f(C)\,|\,C\in\CC_{(\alpha)}\,\right\}.\]
\\
\textit{Relations of type three.} It is the set of all quadratic monomial relations of the form $ab$ in $K\Q_{\G}$ where $ab$ is not a subpath of any special cycle.\\

\noindent We denote by $\rho_{\G}$ the union of these three types of relations.

\begin{defin}\label{009}
Let $K$ be a field and $\G$ a Brauer configuration. The \textit{Brauer configuration algebra} $\L$ \textit{associated to} $\G$ is defined to be $K\Q/I$, where $\Q$ is the quiver induced by $\G$ and $I$ is the ideal in $K\Q$ generated by the set of relations $\rho_{\G}$.
\end{defin}

\begin{obs}\label{016}
As we mentioned in Section \ref{sec02} an element $V\in\G_1$ is a \textit{multiset}, which is a set where the elements in it can appear more than once. So, maybe an expression as $V\cap\G_0$ could have not a very clear meaning because is the intersection of two objects of different nature. But in general we may affirm without any worry that the \textit{intersection between a set and a multiset} is another set, which is formed by the collection of all the elements in common that appear in both the set and the multiset. For example, consider the configuration $\G$ given in Example \ref{002}. In this example the polygon $V_3$ is equal to the multiset $\{1,1,3,3\}$, then according to we just had said we will have that $V_3\cap\G_0=\{1,3\}$.
\end{obs}
%%%En section 1 tenemos los labels de 001 a 009

%%%%Posiblemente agregar un ejemplo aquí, sólo para mostrar el aspecto de algunas relaciones de acuerdo al quiver del Ejemplo \ref{007}.

\section{Basis in $v\L w$}\label{sec3}
A particular case where a vector basis for $v\L w$ is calculated can be found in \cite[Section 3]{mya}, when $v=w$. Actually, we mentioned something about this at the beginning of the introduction (see Proposition \ref{001}). In the present section we consider all the other possible cases, which are those when the vertices $v$ and $w$ in the induced quiver associated to the polygons $V$ and $W$, respectively, are different. In short words, when $V\neq W$. We now refer the reader to \cite[Section 3]{mya} to read the Subsection 3.1 very carefully. In that subsection appears, with all the details, the combinatorial technique that we are going to use to calculate a vector basis for $v\L w$, when $v\neq w$, by adding just a couple of small  technical details. So we stick to the same notation used there to our present task.\\

%We start with an almost obvious lemma.

Let's do a brief remind of the notation and the terminology. Let $\G$ be a Brauer configuration and let $\Q$ be the induced quiver associated to $\G$. For $\alpha\in\G_0$ a fixed vertex with $\val(\alpha)>1$, let $\alpha:V_{i_1}<\cdots<V_{i_{\val(\alpha)}}$ be its successor sequence. Then in the quiver $\Q$ we have a sequence of arrows

\begin{equation}\label{017}
v_{i_1}\stackrel{a^{(\alpha)}_{j_1}}{\longrightarrow} v_{i_2} \stackrel{a^{(\alpha)}_{j_2}}{\longrightarrow} \cdots \stackrel{a^{(\alpha)}_{j_{\val(\alpha)-1}}}{\longrightarrow}v_{i_{\val(\alpha)}}\stackrel{a^{(\alpha)}_{j_{\val(\alpha)}}}{\longrightarrow} v_{i_1}.
\end{equation}
 Now, if $V$ is a polygon which appears in the successor sequence of $\alpha$ and $\occ(\alpha,V)>1$ then the sequence of arrows in (\ref{017}) can be transformed and represented by the following \textit{Special and Non-Special diagram associated to} $\alpha$.
 
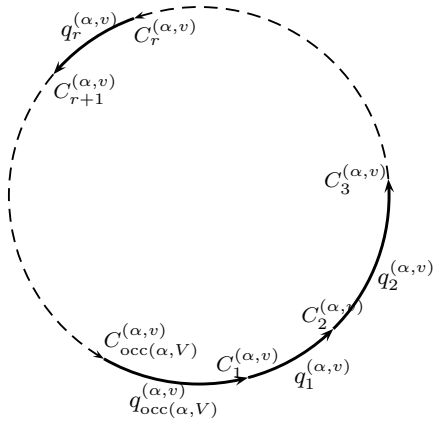
\begin{figure}[H]
\centering
 \begin{pspicture}(-3.2,-3.2)(3.2,3.2)
  \psarc[linewidth=1.1pt]{->}(0,0){2.5}{-120}{-75}%arco1
  \psarc[linewidth=1.1pt]{->}(0,0){2.5}{-75}{-45}%arco2
  \psarc[linewidth=1.1pt]{->}(0,0){2.5}{-45}{5}%arco3
  \psarc[linewidth=0.7pt,linestyle=dashed]{->}(0,0){2.5}{5}{110}%arco4
  \psarc[linewidth=1.1pt]{->}(0,0){2.5}{110}{140}%arco5
  \psarc[linewidth=0.7pt,linestyle=dashed]{->}(0,0){2.5}{140}{240}%arco6
  \rput[t](-0.326,-2.479){\footnotesize$q^{(\alpha,v)}_{\occ(\alpha,V)}$}
  \rput[tl](1.250,-2.165){\footnotesize$q^{(\alpha,v)}_{1}$}
  \rput[tl](2.349,-0.855){\footnotesize$q^{(\alpha,v)}_{2}$}
  \rput[b](-1.434,2.048){\footnotesize$q^{(\alpha,v)}_{r}$}
  \rput[bl](-1.250,-2.165){\scriptsize$C^{(\alpha,v)}_{\occ(\alpha,V)}$}
  \rput[b](0.647,-2.415){\scriptsize$C^{(\alpha,v)}_{1}$}
  \rput[b](1.768,-1.768){\scriptsize$C^{(\alpha,v)}_{2}$}
  \rput[r](2.490,0.218){\scriptsize$C^{(\alpha,v)}_{3}$}
  \rput[tl](-0.855,2.349){\scriptsize$C^{(\alpha,v)}_{r}$}
  \rput[tl](-1.915,1.607){\scriptsize$C^{(\alpha,v)}_{r+1}$}
 \end{pspicture}
 \caption{\small Special and Non-special diagram associated to $\alpha$.}\label{fig3}
\end{figure}
As we can see, the collection $\C^{\,v}_{(\alpha)}=\left\{C^{(\alpha,v)}_{1},\ldots,C^{(\alpha,v)}_{\occ(\alpha,V)}\right\}$ is formed by all the special $\alpha$-cycles at $v$ and the collection $\neg\C^{\,v}_{(\alpha)}=\left\{q^{(\alpha,v)}_{1},\ldots,q^{(\alpha,v)}_{\occ(\alpha,V)}\right\}$ by all the non-special $\alpha$-cycles at $v$.

We divide in \textit{intervals} the diagram in Fig. \ref{fig3} according to the order in which occurs the vertex $v$ on the diagram above\footnote{Remember that after to fix one of the occurrences of the vertex $v$ we label it as \textit{1st} $v$ and then in counter clockwise continue labeling the other occurrences of $v$ as \textit{2nd} $v$, \textit{3rd} $v\ldots$ etc \cite[Pages 296 and 297]{mya}.}. This division is made as follows
\begin{itemize}
\item \textit{1st} $(\alpha,v)$-\textit{interval}: is the segment formed by all the vertices in $\Q$ that appear, in the same order of occurrence and repetitions included, in the special and non-special diagram associated to $\alpha$ between the \textit{1st} $v$ and the \textit{2nd} $v$.

\item \textit{2nd} $(\alpha,v)$-\textit{interval}: is the segment formed by all the vertices in $\Q$ that appear, in the same order of occurrence and repetitions included, in the special and non-special diagram associated to $\alpha$ between the \textit{2nd} $v$ and the \textit{3rd} $v$.
\item And so on$\ldots$
\end{itemize}

When $\alpha$ and $v$ are clear from the context we just call them intervals instead of $(\alpha,v)$-intervals. Now, if in Fig. \ref{fig3} we delete the $q^{(\alpha,v)}_{i}$\,'s and instead of the $C^{(\alpha,v)}_i$\,'s we put the corresponding $i$-\textit{th} $v$'s, and making $s=\occ(\alpha,V)$, the diagram that represents all the intervals would be like in the Fig. \ref{fig4}.

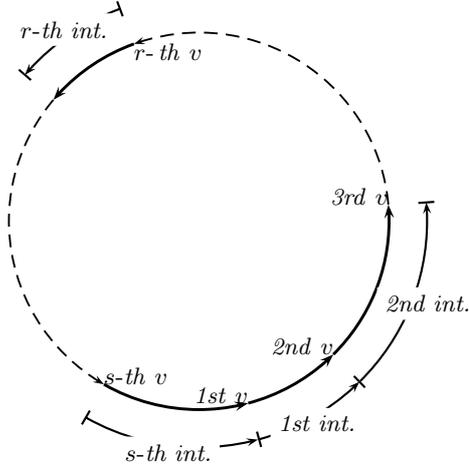
\begin{figure}[H]
\centering
 \begin{pspicture}(-3.2,-3.2)(3.2,3.2)
  \psarc[linewidth=1.1pt]{->}(0,0){2.5}{-120}{-75}%arco1
  \psarc[linewidth=1.1pt]{->}(0,0){2.5}{-75}{-45}%arco2
  \psarc[linewidth=1.1pt]{->}(0,0){2.5}{-45}{5}%arco3
  \psarc[linewidth=0.7pt,linestyle=dashed]{->}(0,0){2.5}{5}{110}%arco4
  \psarc[linewidth=1.1pt]{->}(0,0){2.5}{110}{140}%arco5
  \psarc[linewidth=0.7pt,linestyle=dashed]{->}(0,0){2.5}{140}{240}%arco6
  
  \psarc[linewidth=0.8pt]{|->|}(0,0){3}{-75}{-45}%arco2'
  \psarc[linewidth=0.8pt]{->|}(0,0){3}{-45}{5}%arco3'
  \psarc[linewidth=0.8pt]{|->|}(0,0){3}{110}{140}%arco5'
  \psarc[linewidth=0.8pt]{|->}(0,0){3}{-120}{-75}%arco1'
  
  \rput*(1.55,-2.685){\footnotesize\textit{1st int.}}
  \rput*(3.007,-1.094){\footnotesize\textit{2nd int.}}
  \rput*(-1.778,2.539){\footnotesize$r$-\textit{th int.}}
  \rput*(-0.405,-3.073){\footnotesize$s$-\textit{th int.}}
  
  \rput[br](0.647,-2.415){\footnotesize\textit{1st} $v$}
  \rput[br](1.768,-1.768){\footnotesize\textit{2nd} $v$}
  \rput[br](2.490,0.218){\footnotesize\textit{3rd} $v$}
  \rput[tl](-0.855,2.349){\footnotesize$r$-\,\textit{th} $v$}
  \rput[bl](-1.250,-2.165){\footnotesize$s$-\textit{th} $v$}
 \end{pspicture}
 \caption{\small The $(\alpha,v)$-intervals}\label{fig4}
\end{figure}

Now, if $W\in\V_{(\alpha)}$, where $\V_{(\alpha)}=\left\{W\in\G_0\,|\,\alpha\in W\right\}$, with $W\neq V$, and $w$ is the vertex in $\Q$ associated to $W$, we denote by $\occ^{(\alpha,v)}_i(w)$ the number of times that the vertex $w$ appears in the $i$-th interval, for every $1\le i\le\occ(\alpha,V)$. It is clear that $\occ^{(\alpha,v)}_i(w)\ge0$, for all $1\le i\le\occ(\alpha,V)$, and \begin{equation}\label{024}\sum_{i=1}^{\occ(\alpha,V)}\occ^{(\alpha,v)}_i(w)=\occ(\alpha,W).\end{equation}

The following lemma is almost obvious and easy to prove. We left the proof to the reader.

\begin{lem}\label{015}
Let $\L=K\Q/I$ be the Brauer configuration algebra associated to the Brauer configuration $\G$. Let $V,W$ be polygons in $\G_1$, and let $v,w$ be the respective associated vertices in $\Q$. Then \[v\L w\neq\{0\}\iff\overline{V}\cap\overline{W}\neq\emptyset,\] where $\overline{V}=V\cap\G_0$ and $\overline{W}=W\cap\G_0$.
\end{lem}

\begin{propo}\label{018}
Let $\L=K\Q/I$ be the Brauer configuration algebra induced by $\G=(\G_0,\G_1,\mu,\oo)$. If $V,W\in\G_1$, with $V\neq W$, then \[\textrm{dim}_Kv\L w=\sum_{\alpha\in\overline{V}\cap\overline{W}}\mu(\alpha)\occ(\alpha,V)\occ(\alpha,W),\] where $v$ and $w$ are the vertices in $\Q$ associated to $V$ and $W$ respectively, and $\overline{V}=\G_0\cap V$, $\overline{W}=\G_0\cap W$.
\end{propo}
\begin{proof}
Let $V,W\in\G_1$ be two different polygons of the configuration, and let $v$ and $w$ the vertices in $\Q$ associated to $V$ and $W$ respectively. By Lemma \ref{015} is clear that dim$_Kv\L w=0\iff\overline{V}\cap\overline{W}=\emptyset$. So, let $\alpha$ be a vertex in $\overline{V}\cap\overline{W}$ and let's suppose first that $\occ(\alpha,V)=1$ and $\occ(\alpha,W)\ge1$. In this case we have an only special $\alpha$-cycle at $v$, given by $C^{(\alpha,v)}_1$. Let $w$ be a fixed vertex of those that appear in the special and non-special diagram associated to $\alpha$. If we denote by $q_{(v,w)}$ the composition of all the arrows in the special and non-special diagram between $v$ and this $w$, we obtain $\mu(\alpha)$ paths in $\Q$ from $v$ to $w$ given by $\left(C^{(\alpha,v)}_1\right)^kq_{(v,w)}$, for $0\le k<\mu(\alpha)$. Now, if we do the same for the rest of $w$'s in the special and non-special diagram associated to $\alpha$ we obtain a total of $\mu(\alpha)\occ(\alpha,W)=\mu(\alpha)\occ(\alpha,V)\occ(\alpha,W)$ paths from $v$ to $w$ in $\Q$ associated to the vertex $\alpha$.

Now, suppose that $\occ(\alpha,V)>1$, $\occ(\alpha,W)\ge1$, and set $s=\occ(\alpha,V)$. Without loss of generality, also suppose that  $\occ^{(\alpha,v)}_s(w)>0$, i.e, there is at least one $w$ in the $s$-th interval in the special and non-special diagram associated to $\alpha$. Let $w$ be one of the occurrences of this vertex in the $s$-th interval, and let $q_{(v,w)}$ be the composition of all the arrows in the $s$-th interval between the $s$-\textit{th} $v$ and $w$. This is represented in Fig. \ref{fig2} 
\begin{figure}[H]
\centering
 \begin{pspicture}(-3.2,-3.2)(3.2,3.2)
  \psarc[linewidth=1.1pt]{->}(0,0){2.5}{-150}{-135}%C1
  \psarc[linewidth=1.1pt]{->}(0,0){2.5}{-135}{-120}%C2
  \psarc[linewidth=1.1pt,linestyle=dotted]{-}(0,0){2.5}{-120}{-100}%C3
  \psarc[linewidth=1.1pt]{->}(0,0){2.5}{-100}{-85}%C4
  \psarc[linewidth=1.1pt]{->}(0,0){2.5}{-85}{-70}%C5
  \psarc[linewidth=1.1pt,linestyle=dotted]{-}(0,0){2.5}{-70}{-50}
  \psarc[linewidth=1.1pt]{->}(0,0){2.5}{-50}{-35}
  \psarc[linewidth=1.1pt]{->}(0,0){2.5}{-35}{-20}
  \psarc[linewidth=0.7pt,linestyle=dashed]{->}(0,0){2.5}{-20}{225}
  \rput[bl](-1.768,-1.768){\footnotesize$s$-\textit{th} $v$}
  \rput[tr](-1.768,-1.768){$\,\,v\,\,$}
  \rput[t](0.227,-2.590){$\,\,w\,\,$}
  \rput[br](2.048,-1.434){\footnotesize\textit{1st} $v$}
  \rput[tl](2.048,-1.434){$\,v\,$}
  \psarc[linewidth=1.01pt]{|->|}(0,0){3}{-135}{-85}
  \rput*(-1.026,-2.819){\small$q_{(v,w)}$}
  \end{pspicture}
  \caption{\small Here $q_{(v,w)}$ represents the composition of all the arrows in the special and non-special diagram associated to $\alpha$ between the $s$-\textit{th} $v$ and the fixed vertex $w$ in the $s$-th interval.}\label{fig2}
 \end{figure}
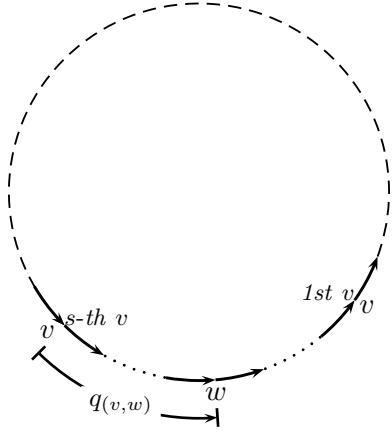
 Then we have that all the possible paths in $\Q$ associated to $\alpha$ that start at $v$, any $v$ of the special and non-special diagram, and finish at this vertex $w$ can be listed as
 \[\begin{array}{l}
 \left(C^{(\alpha,v)}_{s}\right)^kq_{(v,w)}\\
 \left(C^{(\alpha,v)}_{s-1}\right)^kq^{(\alpha,v)}_{s-1}q_{(v,w)}
 \end{array}
 \]
 
 \[\vdots\]
 
 \[\begin{array}{l}
 \left(C^{(\alpha,v)}_{2}\right)^kq^{(\alpha,v)}_{2}\cdots q^{(\alpha,v)}_{s-1}q_{(v,w)}\\
 \left(C^{(\alpha,v)}_{1}\right)^kq^{(\alpha,v)}_{1}\cdots q^{(\alpha,v)}_{s-1}q_{(v,w)}
 \end{array}\]
 where $0\le k<\mu(\alpha)$. As we can see this list contains $\mu(\alpha)\occ(\alpha,V)$ elements. Now, if we do the same for the rest of the $w$'s that are in the $s$-th interval we will obtain $\mu(\alpha)\occ(\alpha,V)\occ^{(\alpha,v)}_s(w)$ paths in $\Q$ associated to $\alpha$ that start at $v$ and finish at $w$, where $w$ is a vertex in the $s$-th interval. Finally, doing this very same reasoning for the rest of the intervals in the special and non-special diagram associated to $\alpha$ we can conclude that the total number of paths in $\Q$ associated to the vertex $\alpha$ that start at $v$ and finish at $w$ is equal to
 \begin{eqnarray*}
 \sum_{i=1}^{s}\mu(\alpha)\occ(\alpha,V)\occ^{(\alpha,v)}_i(w) & = & \mu(\alpha)\occ(\alpha,V)\sum_{i=1}^{s}\occ^{(\alpha,v)}_i(w)\\
  & = & \mu(\alpha)\occ(\alpha,V)\occ(\alpha,W).
 \end{eqnarray*}
 Hence, we have that \[\textrm{dim}_Kv\L w=\sum_{\alpha\in\overline{V}\cap\overline{W}}\mu(\alpha)\occ(\alpha,V)\occ(\alpha,W).\]
\end{proof}
\section{The Cartan matrix of an algebra}

Let $\L$ be a finite-dimensional $K$-algebra (where $\L$ is not necessarily a Brauer configuration algebra) and $P_1,\ldots, P_n$ be the pairwise non-isomorphic list of the projective (right) $\L$-modules. As usual, we denote by $\top(P_i)$ the quotient $P_i/\rad\,P_i$, for every $1\le i\le n$. It is very known that all the $\L$-modules $\top(P_i)$ are simple, and every simple $\L$-module is isomorphic to one of these. For each $1\le i\le n$ we denote by $S_i$ any simple $\L$-module isomorphic to $\top(P_i)$.

\begin{defin}\label{010}
Let $\CM_{\L}=\left(c_{i,j}\right)$ be the $n\times n$ matrix with entries in $\mathbb{Z}_{\ge0}$, where $c_{i,j}$ is the multiplicity of the simple module $\top(P_i)$ in a composition series of $P_j$. We call this matrix the \textit{Cartan matrix} of the algebra $\L$. The integers $c_{i,j}$ are called \textit{Cartan numbers} of the algebra $\L$.

%We define the \textit{Cartan matrix} $C_{\L}=\left(c_{i,j}\right)$ of $\L$ as the $n\times n$ matrix with entries in $\mathbb{Z}$ given by \[c_{i,j}=\textrm{dim}_K\hom_K(P_j,P_i),\] for every $1\le i, j\le n$.
\end{defin}
There is a different particular way to define the Cartan matrix of an algebra. For this we need the following lemma.

\begin{lem}\label{012}
Let $M$ be a $\L$-module of finite length and for every $1\le i\le n$ let $\Delta_i$ denote the division ring given by $\hom_{\L}(S_i,S_i)=\enD_{\L}(S_i)$. Then the value dim$_{\Delta_i}\hom_{\L}(P_i,M)$ coincides with the multiplicity of $S_i$ in a composition factor of $M$.
\end{lem}

\begin{proof}
In \cite[Page 14]{ben} you can find a very nice proof of this lemma.
\end{proof}
Using Lemma \ref{012} we have that if $\CM_{\L}=\left(c_{i,j}\right)$ is the Cartan matrix of $\L$ then
\begin{equation}\label{013}
c_{i,j}=\textrm{dim}_{\Delta_i}\hom_{\L}(P_i,P_j).
\end{equation}

\begin{obs}\label{014}
When the field $K$ is algebraically closed every division ring $\Delta_i=\enD_{\L}(S_i)$ must be equal to $K$. This is due to the fact that there are no finite dimensional division rings over $K$, except the only case of $K$ itself.
\end{obs}
The following lemma is a very known result (see for example \cite[Lemma 4.2]{rep1}).

\begin{lem}\label{011}
Let $M$ be a $\L$-module and $e\in\L$ be an idempotent. Then \[Me\cong\hom_{\L}(e\L,M)\] as right $e\L e$-modules. In particular, is also an isomorphism of vector $K$-spaces.
\end{lem}

So, if we have that $f\in\L$ is also an idempotent then by Lemma \ref{011} it follows that \[\hom_{\L}(e\L,f\L)\cong f\L e.\] Now, let $\left\{e_1,\ldots,e_n\right\}$ be a complete set of primitive orthogonal idempotents of $\L$. We can assume then that $P_i=e_i\L$, for every $1\le i\le n$. We also are going to assume from now on that the field $K$ is algebraically closed, then the Cartan numbers in (\ref{013}) are given now by
\begin{eqnarray}
c_{i,j} & = & \textrm{dim}_K\hom_{\L}(P_i,P_j)\nonumber\\
 & = & \textrm{dim}_K\hom_{\L}(e_i\L,e_j\L)\nonumber\\
 & = & \textrm{dim}_Ke_j\L e_i.\label{019}
\end{eqnarray}
According to \cite{rep1} we remind that if $M$ is a finitely generated (right) $\L$-module its \textit{dimension vector} is defined as the vector given by \[\textrm{\bf dim}\,M=\left(\textrm{dim}_KMe_1,\ldots,\textrm{dim}_KMe_n\right)^t\]
Then, by using the equality in (\ref{019}) we have that the Cartan matrix of $\L$ can be expressed as\[\CM_{\L}=\left(\,\textrm{\bf dim}\,P_1\,\cdots\,\textrm{\bf dim}\,P_n\,\right).\] Now, if $\L$ is a \textit{symmetric} algebra we have that the contravariant functors Hom$_K(-,K)$ and Hom$_{\L}(-,\L)$ are isomorphic functors (see \cite[Proposition IV.3.8]{aus}). Defining $I_i=\textrm{Hom}_K(\L e_i,K)$, for each $1\le i\le n$, we know that every indecomposable injective (right) $\L$-module is isomorphic to one of these $\L$-modules, then if $\L$ is symmetric we obtain respectively that
\begin{eqnarray*}
I_i & = & \textrm{Hom}_K(\L e_i,K)\\
 & \cong & \textrm{Hom}_{\L}(\L e_i,\L)\\
 & \cong & e_i\L\\
 & = & P_i.
\end{eqnarray*}
We have the following proposition.
\begin{propo}\label{020}
Let $\L$ be a symmetric finite dimensional algebra. Then the Cartan matrix of $\L$ is symmetric.
\end{propo}
\begin{proof}
Let $\CM_{\L}$ denote the Cartan matrix of $\L$. By \cite[Proposition III.3.8]{rep1} we have the equalities
\begin{eqnarray*}
\textrm{\bf dim}\,P_i & = & \CM_{\L}\cdot\textrm{\bf dim}\,S_i,\\\textrm{\bf dim}\,I_i & = & \CM_{\L}^t\cdot\textrm{\bf dim}\,S_i,
\end{eqnarray*}
then by the isomorphism above it follows that $\textrm{\bf dim}\,P_i=\textrm{\bf dim}\,I_i$ and hence \[\CM_{\L}\cdot\textrm{\bf dim}\,S_i=\CM_{\L}^t\cdot\textrm{\bf dim}\,S_i.\] This equality is satisfied for all $1\le i\le n$, then we can conclude that \[\CM_{\L}=\CM_{\L}^t.\]
\end{proof}

\section{The Cartan matrix of a Brauer configuration algebra}

From all our previous work we are now ready to give right away the explicit expression of the Cartan matrix of a Brauer configuration algebra. Let $\L=K\Q/I$ be the Brauer configuration algebra associated to $\G=(\G_0,\G_1,\mu,\oo)$, where $K$ is an algebraically closed field, and let $\CM_{\L}=\left(c_{v,w}\right)_{V,W\in\G_1}$ be the Cartan matrix of $\L$. Then by \cite[Proposition 3.3]{mya} and Proposition \ref{018} we have that
\begin{equation}\label{021}
c_{v,w}=\left\{\begin{array}{cr}2+\sum\limits_{\alpha\in\overline{V}}\occ(\alpha,V)\left(\occ(\alpha,V)\mu(\alpha)-1\right), & V=W;\\\sum\limits_{\alpha\in\overline{V}\cap\overline{W}}\mu(\alpha)\occ(\alpha,V)\occ(\alpha,W), & V\neq W.\end{array}\right.
\end{equation}
By \cite[Proposition 3.2]{brau} we have that $\L$ is a symmetric algebra, then by Proposition \ref{020} the matrix $\CM_{\L}$ is symmetric, but this is clearly seen without using Proposition \ref{020}. Just watch the expression for the dimension in Proposition \ref{018}.

\begin{ejem}\label{022}
We are going to consider a particular class of Brauer configuration algebras that has a very interesting class of Cartan matrices. Let $\G=(\G_0,\G_1,\mu,\oo)$ be the Brauer configuration such that $\mu\equiv1$ and the polygons in $\G_1=\left\{V_1,\ldots,V_n\right\}$ are sets, with $n\ge2$. Because of this, in particular, we can affirm that for any $V\in\G_1$ \[\occ(\alpha,V)\le1,\textrm{ for all }\alpha\in\G_0.\] Let $\L$ be the Brauer configuration algebra induced by $\G$ and let $\CM_{\L}$ be its Cartan matrix, then by the expression in (\ref{021}) we obtain that
\[\CM_{\L}=\left(\begin{array}{cccc}2 & |V_1\cap V_2| & \cdots & |V_1\cap V_n|\\|V_1\cap V_2| & 2 &  & |V_2\cap V_n|\\ \vdots & \vdots & \ddots & \vdots\\|V_1\cap V_n| & |V_2\cap V_n| & \cdots & 2\end{array}\right)\]
\end{ejem}

\begin{ejem}\label{023}
Let $\G$ be the Brauer configuration from Example \ref{002}. It is not difficult to see that 
\begin{multicols}{3}
\begin{enumerate}[]
\item $\overline{V}_1\cap\overline{V}_2=\{1,2\}$,
\item $\overline{V}_1\cap\overline{V}_3=\{1\}$,
\item $\overline{V}_1\cap\overline{V}_4=\emptyset$,
\item $\overline{V}_2\cap\overline{V}_3=\{1\}$,
\item $\overline{V}_2\cap\overline{V}_4=\emptyset$,
\item $\overline{V}_3\cap\overline{V}_4=\{3\}$.
\end{enumerate}
\end{multicols}
Let $\L$ be the Brauer configuration algebra associated to $\G$ and let $\CM_{\L}=\left(c_{i,j}\right)_{1\le i,j\le 4}$ be the respective Cartan matrix, then the entries of the main diagonal are equal to
\[
\begin{array}{rclcl}
c_{1,1} & = & 2+1\times(1\times2-1)+1\times(1\times2-1) & = &4\\
 c_{2,2} & = & 2+1\times(1\times2-1)+1\times(1\times2-1) & = &4\\
 c_{3,3} & = & 2+2\times(2\times2-1)+2\times(1\times2-1) & = & 10\\
  c_{4,4} & = & 2 & & 
\end{array}
\]
and the other entries are given by
\[
\begin{array}{rclcl}
c_{1,2} & = & 2\times1\times1+2\times1\times1 & = & 4\\
c_{1,3} & = & 2\times1\times2 & = & 4\\
c_{1,4} & = & 0 & &\\
c_{2,3} & = & 2\times1\times2 & = & 4\\
c_{2,4} & = & 0 & &\\
c_{3,4} & = & 1\times2\times1 & = & 2
\end{array}
\]
then we have that the Cartan matrix is equal to
\[\CM_{\L}=\left(\begin{array}{cccc}4 & 4 & 4 & 0\\4 & 4 & 4 & 0\\4 & 4 & 10 & 2\\0 & 0 & 2 & 2\end{array}\right)\]
It is a known result that the summation of all the entries of the Cartan matrix of an algebra coincides with the vector dimension of the algebra. So, let's check this result in this example just to be a little more confident about the veracity of the principal result of the present work. By \cite[Proposition 3.13]{brau} we can calculate the vector dimension of $\L$ by using the combinatorial information\footnote{The expression to compute this value is equal to $2|\G_1|+\sum\limits_{\alpha\in\G_0}\val(\alpha)(\mu(\alpha)\val(\alpha)-1)$.} in $\G$. Then this value for $\L$ is 
{\small
\begin{eqnarray*}
\textrm{dim}_k\L & = & 2\times4+4\times(2\times4-1)+2\times(2\times2-1)+3\times(1\times3-1)\\
 & = & 8+28+6+6\\
 & = & 48
\end{eqnarray*}
}
By the other side, it's easily seen that the summation of all the entries of $\CM_{\L}$ is equal to
\begin{eqnarray*}
8\times4+3\times2+10 & = & 32+6+10\\
 & = & 48
\end{eqnarray*}
%Apparently the result is fine.
\end{ejem}

\end{document}